\documentclass[12pt,twoside]{amsart}
\usepackage{amsmath}
\usepackage{amsthm}
\usepackage{amsfonts}
\usepackage{amssymb}
\usepackage{latexsym}
\usepackage[all]{xy}
\usepackage{epsfig}
\usepackage{amscd,mathrsfs,graphicx,mathrsfs}

\date{}
\allowdisplaybreaks[4] \footskip=15pt

\renewcommand{\uppercasenonmath}[1]{}

\numberwithin{equation}{section} \theoremstyle{plain}
\newtheorem*{theorem*}{Main Theorem}
\newtheorem{theorem}{Theorem}[section]
\newtheorem{corollary}[theorem]{Corollary}
\newtheorem*{corollary*}{Corollary}
\newtheorem{lemma}[theorem]{Lemma}
\newtheorem*{lemma*}{Lemma}
\newtheorem{proposition}[theorem]{Proposition}
\newtheorem*{proposition*}{Proposition}

\newtheorem*{remark*}{Remark}

\newtheorem*{example*}{Example}

\newtheorem*{definition*}{Definition}

\newtheorem*{conjecture*}{Conjecture}
\newtheorem*{ack*}{ACKNOWLEDGEMENTS}

\newcommand{\pf}{\noindent\begin {proof}}
\newcommand{\epf}{\end{proof}}

\begin{document}
\begin{center}
{\Large \bf Quillen equivalence of singular model categories
 \footnotetext{
E-mail address: wren@cqnu.edu.cn.}}

\vspace{0.5cm}    Wei Ren\\
{School of Mathematical Sciences, Chongqing Normal University,\\ Chongqing 401331, China
}
\end{center}


\bigskip
\centerline { \bf  Abstract}
\leftskip10truemm \rightskip10truemm
Let $R$ be a left-Gorenstein ring. We show that there is a Quillen equivalence between singular contraderived model category and singular coderived model category. Consequently, an equivalence between the homotopy category of exact complexes of projective modules and the homotopy category of exact complexes of injective modules is given.
\bigskip


\leftskip0truemm \rightskip0truemm \vbox to 0.2cm{}

\section { \bf Introduction}
The notion of recollement of triangulated categories was introduced with an ideal that one category can be viewed as being ``glue together'' from two others. Denote by $\mathbf{K}(\mathcal{I})$ (resp. $\mathbf{K}_{ex}(\mathcal{I})$) the homotopy category of complexes (resp. exact complexes) of injective modules, and $\mathbf{D}(R)$ the derived category. In \cite{Kra05}, a recollement
$\xymatrix{
\mathbf{K}_{ex}(\mathcal{I})\ar[r]^{} & \mathbf{K}(\mathcal{I}) \ar[r]^{}\ar@<-0.6ex>[l]^{} \ar@<0.6ex>[l]_{} &\mathbf{D}(R)\ar@<-0.6ex>[l]^{} \ar@<0.6ex>[l]_{}}$ was established. A dual one  $\xymatrix{
\mathbf{K}_{ex}(\mathcal{P})\ar[r]^{} & \mathbf{K}(\mathcal{P}) \ar[r]^{}\ar@<-0.6ex>[l]^{} \ar@<0.6ex>[l]_{} &\mathbf{D}(R)\ar@<-0.6ex>[l]^{} \ar@<0.6ex>[l]_{}}$
for homotopy categories with respect to projective modules was obtained in \cite{Bec14}.

Equivalences between the homotopy categories $\mathbf{K}(\mathcal{I})$ and $\mathbf{K}(\mathcal{P})$, which appear in the middle of the above mentioned recollements, were proved under different conditions. For example, let $R$ be a commutative noetherian ring with a dualizing complex $D$, then  there is a triangle-equivalence $D\otimes_{R}-: \mathbf{K}(\mathcal{P})\rightarrow \mathbf{K}(\mathcal{I})$; see \cite[Theorem I]{IK06}. If $R$ is a left-Gorenstein ring \cite{Bel00} (i.e. a ring such that any left $R$-module has finite projective dimension if and only if it has finite injective dimension), an equivalence $\mathbf{K}(\mathcal{P})\simeq\mathbf{K}(\mathcal{I})$ was established in \cite{Chen10} via relative derived categories with respect to the balanced pairs.

It is natural to ask that for the homotopy categories of exact complexes appeared on the left of the above recollements, when does the equivalence $\mathbf{K}_{ex}(\mathcal{P})\simeq\mathbf{K}_{ex}(\mathcal{I})$ hold? It is worth noting that one cannot restrict the equivalent functors in \cite{Chen10, IK06} to get an answer.

Recall that a model structure on an abelian category is three distinguished classes of maps, called weak equivalences, cofibrations and fibrations respectively, satisfying a few axioms. For a model category, the associated homotopy category is constructed by localization with respect to weak equivalences. The category $\mathbf{K}_{ex}(\mathcal{P})$ (resp. $\mathbf{K}_{ex}(\mathcal{I})$) was realized  as the homotopy category of singular contraderived (resp. singular coderived) model category, see \cite{Bec14, Gil08}.

In this paper, we are inspired to show a Quillen equivalence between singular contraderived model category and singular coderived model category, and then we give an equivalence $\mathbf{K}_{ex}(\mathcal{P})\simeq\mathbf{K}_{ex}(\mathcal{I})$. To state the main result of the paper, we need some notations. Let $\mathrm{Ch}(R)$ be the category of $R$-complexes. Following \cite{Gil08}, let $ex\widetilde{\mathcal{P}}$ (resp. $ex\widetilde{\mathcal{I}}$) be the subcategory of all exact complexes of projective (resp. injective) modules, and  $(ex\widetilde{\mathcal{P}})^{\perp}$ (resp. ${^{\perp}}(ex\widetilde{\mathcal{I}})$) be the right (resp. left) orthogonal. In the language of Hovey's correspondence \cite[Theorem 2.2]{Hov02}, the singular contraderived model structure on $\mathrm{Ch}(R)$ is denoted by a triple $\mathcal{M}_{sing}^{ctr} = (ex\widetilde{\mathcal{P}}, (ex\widetilde{\mathcal{P}})^{\perp}, \mathrm{Ch}(R))$, and the singular coderived model structure is denoted by $\mathcal{M}_{sing}^{co} = (\mathrm{Ch}(R), ^{\perp}(ex\widetilde{\mathcal{I}}), ex\widetilde{\mathcal{I}})$.

Let $X$ be any complex. We denote by $\Omega: \mathrm{Ch}(R)\rightarrow \mathrm{Mod}(R)$ the functor given by $\Omega(X)= X_0/\mathrm{Im}d_{1}^{X}$, and $\Theta: \mathrm{Ch}(R)\rightarrow \mathrm{Mod}(R)$ the functor given by $\Theta(X)= \mathrm{Ker}d_{0}^{X}$. We define $\Lambda: \mathrm{Mod}(R)\rightarrow \mathrm{Ch}(R)$ to be a functor which sends every module to a stalk complex concentrated on degree zero.

\begin{theorem}\label{main}
 Let $R$ be a left-Gorenstein ring, $F=\Lambda\Omega$ and $G=\Lambda\Theta$ be functors on $\mathrm{Ch}(R)$. Then $(F, G): (\mathrm{Ch}(R), \mathcal{M}_{sing}^{ctr})\longrightarrow (\mathrm{Ch}(R), \mathcal{M}_{sing}^{co})$ is a Quillen equivalence between the singular contraderived model category and singular coderived model category.
\end{theorem}

By fundamental results on homotopy categories of model categories (see e.g. \cite[Theorem 1.2.10]{Hov99}), one has triangle-equivalences $\mathrm{Ho}(\mathcal{M}_{sing}^{ctr})\simeq \mathbf{K}(\mathcal{P})$ and  $\mathrm{Ho}(\mathcal{M}_{sing}^{co})\simeq \mathbf{K}(\mathcal{I})$; see \cite{Bec14} or Corollary \ref{cor 1} below. Moreover, a Quillen equivalence of model categories yields an adjoint equivalence of corresponding homotopy categories, see e.g. \cite[Proposition 1.3.13]{Hov99}. Hence we can give an affirmative answer to the above question.

\begin{corollary}
Let $R$ be a left-Gorenstein ring. Then there is an equivalence $F^{'}: \mathbf{K}_{ex}(\mathcal{P})\rightarrow \mathbf{K}_{ex}(\mathcal{I})$ which is defined on objects by first taking $F$, and then taking fibrant replacement (= a special $ex\widetilde{\mathcal{I}}$-preenvelope); the inverse $G^{'}: \mathbf{K}_{ex}(\mathcal{I})\rightarrow \mathbf{K}_{ex}(\mathcal{P})$ is defined on objects by first taking $G$, and then taking cofibrant replacement (= a special $ex\widetilde{\mathcal{P}}$-precover).
\end{corollary}

We remark that part of the above equivalence is known. By \cite[Proposition 3.1.4 and  3.1.5]{Bec14}, $\Omega$ is a Quillen equivalence between the singular contraderived model category and Gorenstein projective model category, and $\Theta$ is a Quillen equivalence between the singular coderived model category and Gorenstein injective model categories; see \cite[Theorem 8.6]{Hov02} for details on Gorenstein projective and Gorenstein injective model categories. Analogously, the equivalences were proved in \cite[Theorem 5.8 and 8.8]{BGH} with respect to Gorenstein AC-projective and Gorenstein AC-injective model structures. However, it seems that there is no reference to check the Quillen equivalence in Theorem \ref{main} specifically.

It is well known that over a left-Gorenstein ring $R$, the exact complex of projective (injective) modules is precisely the totally acyclic complex of projective (injective) modules. The equivalence between homotopy category of totally acyclic complexes $\mathbf{K}_{tac}(\mathcal{P})$ of projective modules and singularity category $\mathbf{D}_{sg}(R)$ for an artin ring or a commutative noetherian local ring was studied in \cite{BJO15}. If the base ring is Iwanaga-Gorenstein, it is proved in \cite[Theorem 4.4.1]{Buc} that the syzygy functor $\Omega$ gives an equivalence between homotopy category of (totally) acyclic complexes of projective modules and the stable category of Gorenstein projective modules.

\subsection*{Question} Let $\mathcal{GI}$ and $\mathcal{GP}$ denote the classes of Gorenstein injective and Gorenstein projective modules, respectively. There is a ``Gorenstein version'' of the aforementioned recollements in \cite{Gil16}, i.e. $\xymatrix{
\mathbf{K}_{ex}(\mathcal{GI})\ar[r]^{} & \mathbf{K}(\mathcal{GI}) \ar[r]^{}\ar@<-0.6ex>[l]^{} \ar@<0.6ex>[l]_{} &\mathbf{D}(R)\ar@<-0.6ex>[l]^{} \ar@<0.6ex>[l]_{}}$ and $\xymatrix{
\mathbf{K}_{ex}(\mathcal{GP})\ar[r]^{} & \mathbf{K}(\mathcal{GP}) \ar[r]^{}\ar@<-0.6ex>[l]^{} \ar@<0.6ex>[l]_{} &\mathbf{D}(R)\ar@<-0.6ex>[l]^{} \ar@<0.6ex>[l]_{}}$.

If the underlying ring is left-Gorenstein, it follows from \cite{Chen10} that $\mathbf{K}(\mathcal{GP})\simeq \mathbf{K}(\mathcal{GI})$; we also recovered this equivalence in \cite{Ren19}. However, we do not know if it is true that $\mathbf{K}_{ex}(\mathcal{GP})\simeq \mathbf{K}_{ex}(\mathcal{GI})$. We remark that one can not get an answer by simply restricting the equivalent functor $\mathbf{K}(\mathcal{GP})\simeq \mathbf{K}(\mathcal{GI})$ in \cite{Chen10}, or by the methods in \cite{Ren19}.

\section { \bf The proof of the theorem}

Throughout the paper, let $R$ be a left-Gorenstein ring. All modules are left $R$-modules.

Let $\mathcal{A}$ be an abelian category with enough projectives and injectives. A pair of classes $(\mathcal{X}, \mathcal{Y})$ in $\mathcal{A}$ is a cotorsion pair provided that $\mathcal{X} =  {^\perp}\mathcal{Y}$ and $\mathcal{Y} = \mathcal{X}^{\perp}$, where $^{\perp}\mathcal{Y} = \{X \mid \mathrm{Ext}^{1}_{\mathcal{A}}(X, Y) = 0,~~\forall~~Y\in \mathcal{Y}\}$ and
$\mathcal{X}^{\perp} = \{Y \mid \mathrm{Ext}^{1}_{\mathcal{A}}(X, Y) = 0,~~\forall~~X\in \mathcal{X}\}$.

The cotorsion pair $(\mathcal{X}, \mathcal{Y})$ is complete provided that for any $M\in \mathcal{A}$,
there exist short exact sequences $0\rightarrow Y\rightarrow X \stackrel{f}\rightarrow M \rightarrow 0$ and  $0\rightarrow M\stackrel{g}\rightarrow Y^{'} \rightarrow X^{'} \rightarrow 0$ with $X, X^{'}\in \mathcal{X}$ and $Y, Y^{'}\in \mathcal{Y}$. In this case, for any $N\in \mathcal{X}$, $\mathrm{Hom}_{\mathcal{A}}(N, f): \mathrm{Hom}_{\mathcal{A}}(N, X)\rightarrow \mathrm{Hom}_{\mathcal{A}}(N, M)$ is surjective since $\mathrm{Ext}^{1}_{\mathcal{A}}(N, Y) = 0$, and then $f: X\rightarrow M$ is said to be a special $\mathcal{X}$-precover of $M$. Dually, $g: M\rightarrow Y^{'}$ is called a special $\mathcal{Y}$-preenvelope of $M$.

By \cite[Theorem 2.2]{Hov02}, an abelian model structure on $\mathcal{A}$ is equivalent to a triple $(\mathcal{A}_{c}, \mathcal{A}_{tri}, \mathcal{A}_{f})$ of subcategories, for which $\mathcal{A}_{tri}$ is thick and both $(\mathcal{A}_{c}, \mathcal{A}_{f}\cap \mathcal{A}_{tri})$ and $(\mathcal{A}_{c}\cap \mathcal{A}_{tri}, \mathcal{A}_{f})$ are complete cotorsion pairs; see also \cite[Chapter VIII]{BR07}. In this case, $\mathcal{A}_{c}$ is the class of cofibrant objects, $\mathcal{A}_{tri}$ is the class of trivial objects and $\mathcal{A}_{f}$ is the class of fibrant objects. The model structure is called ``abelian'' since it is compatible with the abelian structure of the category in the following way: (trivial) cofibrations are monomorphisms with (trivially) cofibrant cokernel, (trivial) fibrations are epimorphisms with (trivially) fibrant kernel, and weak equivalences are morphisms which factor as a trivial cofibratin followed by a trivial fibration.

For convenience, we will use the triple $(\mathcal{A}_{c}, \mathcal{A}_{tri}, \mathcal{A}_{f})$ to denote the corresponding model structure. The following is immediate from \cite[Section 2]{Bec14} or \cite[Theorem 4.7]{Gil08}.

\begin{lemma}\label{lem 1}
On the category $\mathrm{Ch}(R)$ of complexes, there is a singular contraderived model structure
$\mathcal{M}_{sing}^{ctr} = (ex\widetilde{\mathcal{P}}, (ex\widetilde{\mathcal{P}})^{\perp}, \mathrm{Ch}(R))$, and a singular coderived model structure $\mathcal{M}_{sing}^{co} = (\mathrm{Ch}(R), {^{\perp}}(ex\widetilde{\mathcal{I}}), ex\widetilde{\mathcal{I}})$.
\end{lemma}

For a bicomplete abelian category $\mathcal{A}$ with the model structure $\mathcal{M} = (\mathcal{A}_{c}, \mathcal{A}_{tri}, \mathcal{A}_{f})$, the associated homotopy category $\mathrm{Ho}(\mathcal{M})$ is constructed by localization with respect to weak equivalences. The homotopy category of an abelian model category is always a triangulated category. There is an equivalence of categories $\mathrm{Ho}(i): \mathcal{A}_{cf}/\omega = {\mathcal{A}_{cf}/\sim}\rightarrow \mathrm{Ho}(\mathcal{M})$ induced by the inclusion functor $i: \mathcal{A}_{cf}\rightarrow \mathcal{A}$, where $\mathcal{A}_{cf}=\mathcal{A}_{c}\cap \mathcal{A}_{f}$, $f\sim g: M\rightarrow N$ if $g-f$ factors through an object in $\omega = \mathcal{A}_{c}\cap \mathcal{A}_{tri}\cap \mathcal{A}_{f}$; see e.g. \cite[Section 1.2]{Hov99}.

\begin{corollary}\label{cor 1}
There are equivalences $\mathrm{Ho}(\mathcal{M}_{sing}^{ctr})\simeq \mathbf{K}_{ex}(\mathcal{P})$ and
$\mathrm{Ho}(\mathcal{M}_{sing}^{co})\simeq \mathbf{K}_{ex}(\mathcal{I})$.
\end{corollary}

\begin{proof}
We use $\widetilde{\mathcal{P}}$ (resp. $\widetilde{\mathcal{I}}$) to denote the subcategory of contractible complexes of projective (resp. injective) modules. It is well known that a complex $P\in \widetilde{\mathcal{P}}$ if and only if $P$ is exact and each $\mathrm{Ker}d_{i}^{P}$ is a projective module; similarly, complexes in $\widetilde{\mathcal{I}}$ are characterized. Note that for any chain maps $f$ and $g$, if $g-f$ factors through a complex in $\widetilde{\mathcal{P}}$ (or,  a complex in $\widetilde{\mathcal{I}}$), then $f$ is chain homotopic to $g$, denoted by $f\sim g$. Since $ex\widetilde{\mathcal{P}}\cap (ex\widetilde{\mathcal{P}})^{\perp}=\widetilde{\mathcal{P}}$ and $ex\widetilde{\mathcal{I}}\cap {^{\perp}}(ex\widetilde{\mathcal{I}}) = \widetilde{\mathcal{I}}$, the equivalences hold directly.
\end{proof}

Let $F=\Lambda\Omega$ and $G=\Lambda\Theta$ be functors on $\mathrm{Ch}(R)$, where $\Omega$ and  $\Theta$ are functors from $\mathrm{Ch}(R)$ to $\mathrm{Mod}(R)$ such that for any $X\in \mathrm{Ch}(R)$, $\Omega(X)= X_0/\mathrm{Im}d_{1}^{X}$ and $\Theta(X)= \mathrm{Ker}d_{0}^{X}$. Let $\Lambda: \mathrm{Mod}(R)\rightarrow \mathrm{Ch}(R)$ be a functor which sends every module to a stalk complex concentrated on degree zero.

\begin{lemma}\label{lem 2}
Let $X$, $Y$ be any $R$-complexes, and $f: X\rightarrow Y$ a monomorphism of complexes. If $f$ is a quasi-isomorphism, then $\Omega(f)$ is also a monomorphism of $R$-modules.
\end{lemma}

\begin{proof}
We consider the following commutative diagram
$$\xymatrix{
0\ar[r] & \mathrm{Ker}d_{0}^{X} / \mathrm{Im}d_{1}^{X} \ar[r]\ar[d] & X_{0} / \mathrm{Im}d_{1}^{X} \ar[r]\ar[d]_{\Omega(f)}
& X_{0} / \mathrm{Ker}d_{0}^{X}\ar[r]\ar[d] &0\\
0\ar[r] & \mathrm{Ker}d_{0}^{Y} / \mathrm{Im}d_{1}^{Y} \ar[r] & Y_{0} / \mathrm{Im}d_{1}^{Y} \ar[r] & Y_{0} / \mathrm{Ker}d_{0}^{Y}\ar[r] &0 }$$
Since $f$ is a quasi-isomorphism, we have an isomorphism induced by $f$:
$$\mathrm{H}_0(f): \mathrm{H}_0(X)=\mathrm{Ker}d_{0}^{X} / \mathrm{Im}d_{1}^{X}\longrightarrow \mathrm{Ker}d_{0}^{Y} / \mathrm{Im}d_{1}^{Y}=\mathrm{H}_0(Y).$$
The chain map $f$ is monic, then the induced map of modules $X_{0} / \mathrm{Ker}d_{0}^{X}\cong \mathrm{Im}d_{0}^{X}\longrightarrow \mathrm{Im}d_{0}^{Y}\cong Y_{0} / \mathrm{Ker}d_{0}^{Y}$ is also monic. Hence, by the ``Five Lemma'' for the above diagram, we get that $\Omega(f): X_{0} / \mathrm{Im}d_{1}^{X}\longrightarrow Y_{0} / \mathrm{Im}d_{1}^{Y}$ is a monomorphism. We mention that it is also direct to check injectivity of $\Omega(f)$ by diagram chasing.
\end{proof}

For model categories $\mathcal{C}$ and $\mathcal{D}$, recall that an adjunction $(F, G): \mathcal{C}\rightarrow \mathcal{D}$ is a Quillen adjunction if $F$ is a left Quillen functor, or equivalently $G$ is a right Quillen functor. That is, $F$ preserves cofibrations and trivial cofibrations, or $G$ preserves fibrations and trivial fibrations.

\begin{proposition}\label{prop 1}
$(F, G): (\mathrm{Ch}(R), \mathcal{M}_{sing}^{ctr})\rightarrow (\mathrm{Ch}(R), \mathcal{M}_{sing}^{co})$ is a Quillen adjunction.
\end{proposition}

\begin{proof}
Let $X$, $Y$ be any $R$-complexes. It follows from \cite[Lemma 3.1]{Gil04} that $(\Omega, \Lambda): \mathrm{Ch}(R)\rightarrow \mathrm{Mod}(R)$
and  $(\Lambda, \Theta): \mathrm{Mod}(R)\rightarrow \mathrm{Ch}(R)$ are adjunctions. Then we have the following natural isomorphisms:
$\mathrm{Hom}_{\mathrm{Ch}(R)}(F(X), Y)\cong \mathrm{Hom}_{R}(\Omega(X), \Theta(Y))\cong \mathrm{Hom}_{\mathrm{Ch}(R)}(X, G(Y))$.
This implies that $(F, G):  \mathrm{Ch}(R)\rightarrow \mathrm{Ch}(R)$ is an adjunction.

It suffices to show that $F$ preserves cofibration and trivial cofibration. Let $f: X\rightarrow Y$ be a cofibration in $\mathcal{M}_{sing}^{ctr}$, i.e. $f$ is a monomorphism with $\mathrm{Coker}f \in ex\widetilde{\mathcal{P}}$. This yields that $f$ is a quasi-isomorphism, and by Lemma \ref{lem 2}, $\Omega(f)$ is monic. Then, we have an exact sequence $$0\longrightarrow F(X)\stackrel{F(f)}\longrightarrow F(Y)\longrightarrow F(\mathrm{Coker}f)\longrightarrow 0.$$ Since every complex is a cofibrant object in $\mathcal{M}_{sing}^{co}$, this implies that $F(f)$ is a cofibration.

Now suppose $f: X\rightarrow Y$ is a trivial cofibration in $\mathcal{M}_{sing}^{ctr}$, i.e. $f$ is a monomorphism with $\mathrm{Coker}f \in ex\widetilde{\mathcal{P}}\cap (ex\widetilde{\mathcal{P}})^{\perp}= \widetilde{\mathcal{P}}$. Then we have an exact sequence $$0\longrightarrow F(X)\stackrel{F(f)}\longrightarrow F(Y)\longrightarrow F(\mathrm{Coker}f)\longrightarrow 0.$$
Note that $\Omega(\mathrm{Coker}f)$ is a projective module. For any complex $I\in ex\widetilde{\mathcal{I}}$, it is easy to show that any chain map $F(\mathrm{Coker}f) = \Lambda\Omega(\mathrm{Coker}f)\rightarrow I$ is null homotopic, and then $F(\mathrm{Coker}f)\in {^{\perp}}(ex\widetilde{\mathcal{I}})$. Thus $F(f)$ is a trivial cofibration in $\mathcal{M}_{sing}^{co}$. This completes the proof.
\end{proof}

Recall that a module $M$ is Gorenstein projective if $M$ is a syzygy of a totally acyclic complex of projective modules; and dually, Gorenstein injective modules are defined; see \cite{EJ00}. We use $\mathcal{GP}$ and $\mathcal{GI}$ to denote the classes of Gorenstein projective and Gorenstein injective modules, respectively. It is widely accepted that over a left-Gorenstein ring, $(\mathcal{GP}, \mathcal{W})$ and $(\mathcal{W}, \mathcal{GI})$ are complete cotorsion pairs, where $\mathcal{W}$ is the class of modules with finite projective (injective) dimension. In \cite[Theorem 2.7]{Ren18} we show that the cotorsion pair $(\mathcal{GP}, \mathcal{W})$ is cogenerated by a set, i.e. there exists a set $S$ such that $\mathcal{W} =\{S\}^{\perp}$. This also implies the completeness of $(\mathcal{GP}, \mathcal{W})$, and generalizes the Gorenstein projective model structure of $\mathrm{Mod}(R)$ in \cite[Theorem 8.3 and 8.6]{Hov02} from Iwanaga-Gorenstein rings to left-Gorenstein rings.

\begin{lemma}\label{lem 3}
Let $X$, $Y$ be complexes in $ex\widetilde{\mathcal{P}}$, and $f: X\rightarrow Y$ a chain map. If $F(f)$ is a weak equivalence in $\mathcal{M}_{sing}^{co}$, then $f$ is a weak equivalence in $\mathcal{M}_{sing}^{ctr}$.
\end{lemma}

\begin{proof}
In the model category $(\mathrm{Ch}(R), \mathcal{M}_{sing}^{ctr})$, we can factor $f: X\rightarrow Y$ as a trivial cofibration $i: X\rightarrow Z$ followed by a fibration $p: Z\rightarrow Y$. By Proposition \ref{prop 1}, $F(i)$ is a trivial cofibration in $\mathcal{M}_{sing}^{co}$, and then $F(i)$ is a weak equivalence. Then $F(f) = F(p)F(i)$ is a weak equivalence if and only if so is $F(p)$.

Let $L = \mathrm{Coker}i$ and $K =  \mathrm{Ker}p$. It follows from the exact sequence $0\longrightarrow X\stackrel{i}\longrightarrow Z\longrightarrow L\longrightarrow 0$ that $Z\in ex\widetilde{\mathcal{P}}$, where  $X\in ex\widetilde{\mathcal{P}}$ and $L\in ex\widetilde{\mathcal{P}}\cap (ex\widetilde{\mathcal{P}})^{\perp} = \widetilde{\mathcal{P}}$. Moreover, it follows from the exact sequence $0\longrightarrow K\longrightarrow Z\stackrel{p}\longrightarrow Y\longrightarrow 0$ that $K\in ex\widetilde{\mathcal{P}}$.

Let $M$ be any Gorenstein injective module. Then there exists a totally acyclic complex of injective module, saying $I$, such that $M\cong \Theta(I)$. It follows from \cite[Lemma 4.2]{Gil08} that there is an isomorphism
$$\mathrm{Ext}^{1}_{\mathrm{Ch}(R)}(F(K), I) = \mathrm{Ext}^{1}_{\mathrm{Ch}(R)}(\Lambda\Omega(K), I)\cong \mathrm{Ext}^{1}_{R}(\Omega(K), \Theta(I)).$$
Since $F(p)$ is a weak equivalence, $F(K) = \mathrm{Ker}(F(p)) \in {^{\perp}(ex\widetilde{\mathcal{I}})}$. For any Gorenstein injective module $M$, it yields that $\mathrm{Ext}^{1}_{R}(\Omega(K), M) = 0$. Since $(\mathcal{W}, \mathcal{GI})$ is a cotorsion pair, we get that $\Omega(K)$ is a module of finite projective dimension. Moreover, $\Omega(K)$ is a Gorenstein projective module since $R$ is a left-Gorenstein ring and $K\in ex\widetilde{\mathcal{P}}$ is a totally acyclic complex of projective modules. By \cite[Proposition 10.2.3]{EJ00}, the projective dimension of any Gorenstein projective module is either zero or infinity, so $\Omega(K)$ is a projective module. Considering exact sequences $0\rightarrow \mathrm{Ker}d_{i}^{K}\rightarrow K_{i}\rightarrow \mathrm{Ker}d_{i-1}^{K}\rightarrow 0$ inductively, we can prove each syzygy of $K$ is projective, that is, $K$ is a complex in $\widetilde{\mathcal{P}}$. This implies that $p: Z\rightarrow Y$ is a trivial fibration, and hence $f = pi$ is a weak equivalence, as desired.
\end{proof}

\begin{lemma}\label{lem 4}
Let $Y$ be an exact complex of injective $R$-modules. Then $\varepsilon: FG(Y)\rightarrow Y$ is a weak equivalence in $\mathcal{M}_{sing}^{co}$, where $\varepsilon$ is the counit of the adjoint pair $(F, G)$.
\end{lemma}

\begin{proof}
For $Y$, $G(Y)=\Lambda\Theta(Y)= \cdots\rightarrow 0 \rightarrow \mathrm{Ker}d_0^Y\rightarrow 0\rightarrow\cdots$ is a stalk complex with $\mathrm{Ker}d_0^Y$ concentrated in degree zero. It is easy to see that $FG(Y) = G(Y)$. Then the map $\varepsilon: FG(Y)\rightarrow Y$ is given by a natural embedding $\varepsilon_0: \mathrm{Ker}d_0^Y\rightarrow Y_0$ and $\varepsilon_i =0$ for any $i\neq 0$.
Let $C= \mathrm{Coker}\varepsilon$. Then $C=\cdots\longrightarrow Y_2\stackrel{d_2^Y}\longrightarrow Y_1\stackrel{0}\longrightarrow \mathrm{Im}d_0^Y\stackrel{\iota}\longrightarrow Y_{-1}\stackrel{d_{-1}^Y}\longrightarrow Y_{-2}\longrightarrow\cdots$, where $\iota$ is an embedding. Let $Y_{\sqsupset}=\cdots\rightarrow Y_2\stackrel{d_2^Y}\rightarrow Y_1\rightarrow 0$ be a hard truncation, $D=0\rightarrow \mathrm{Im}d_0^Y\stackrel{\iota}\rightarrow Y_{-1}\stackrel{d_{-1}^Y}\rightarrow Y_{-2}\rightarrow\cdots$. Then there is an exact sequence of complexes $0\longrightarrow Y_{\sqsupset}\longrightarrow C\longrightarrow D\longrightarrow 0$.

Let $E$ be any $R$-complex in $ex\widetilde{\mathcal{I}}$. Since $R$ is left-Gorenstein, then $E$ is totally acyclic, and for any $Y_i$, $\mathrm{Hom}_{R}(Y_i, E)$ is an exact complex. By \cite[Lemma 2.4]{CFH06}, the complex $\mathrm{Hom}_{R}(Y_{\sqsupset}, E)$ is exact.
Note that $D$ is an exact sequence, and then $\mathrm{Hom}_{R}(D, E_i)$ is an exact complex for any $i\in \mathbb{Z}$. By \cite[Lemma 2.5]{CFH06}, the complex $\mathrm{Hom}_{R}(D, E)$ is exact. Moreover, it follows from the short exact sequence
$$0\longrightarrow \mathrm{Hom}_{R}(D, E)\longrightarrow \mathrm{Hom}_{R}(C, E)\longrightarrow \mathrm{Hom}_{R}(Y_{\sqsupset}D, E)\longrightarrow 0$$
that the complex $\mathrm{Hom}_{R}(C, E)$ is exact. This implies that every map from $C$ to any complex in $ex\widetilde{\mathcal{I}}$ is null homotopic. Then $C\in {^{\perp}}ex\widetilde{\mathcal{I}}$. Hence, $\varepsilon: FG(Y)\rightarrow Y$ is a trivial cofibration in $\mathcal{M}_{sing}^{co}$, and moreover, $\varepsilon$ is a weak equivalence.
\end{proof}

\begin{lemma}\label{lem 5}
Let $Y$ be an exact complex of injective $R$-modules. Then $F(q): FQG(Y)\rightarrow FG(Y)$ is a weak equivalence in $\mathcal{M}_{sing}^{co}$, where $q: QG(Y)\rightarrow G(Y)$ is a cofibrant replacement in the model category $(\mathrm{Ch}(R), \mathcal{M}_{sing}^{ctr})$.
\end{lemma}

\begin{proof}
For $Y$, $G(Y) = FG(Y) = \cdots\rightarrow 0 \rightarrow \mathrm{Ker}d_0^Y\rightarrow 0\rightarrow\cdots$. By the completeness of the cotorsion pair $(\mathcal{GP}, \mathcal{W})$, there is an exact sequence of $R$-modules $0\rightarrow W\rightarrow M\rightarrow \mathrm{Ker}d_0^Y\rightarrow0$ with $M\in \mathcal{GP}$ and $W\in \mathcal{W}$. Consider the totally acyclic complex $P$ of $M$, we have a short exact sequence $0\rightarrow K\rightarrow P\stackrel{q}\rightarrow G(Y)\rightarrow 0$, see the following diagram
$$\xymatrix@C=20pt@R=10pt{
K=\cdots \ar[r] &P_{1}\ar[dd]_{=}\ar[r]^{} &K_{0}\ar[dd]\ar@{-->}[rd]^{\pi}\ar[rr]^{} & &P_{-1}\ar[r]\ar[dd]_{=}&P_{-2}\ar[r]\ar[dd]_{=}&\cdots \\
 & & & W\ar@{-->}[ur]\ar@{-->}[dd]^{}\\
P= \cdots \ar[r] &P_{1}\ar[dd]_{}\ar[r]^{} &P_{0}\ar[dd]\ar@{-->}[rd]^{}\ar[rr]^{} & &P_{-1}\ar[r]\ar[dd]_{}&P_{-2}\ar[r]\ar[dd]_{}&\cdots\\
 & & & M\ar@{-->}[ur]\ar@{-->}[dd]^{}\\
G(Y) = \cdots \ar[r] &0\ar[r]^{} &\mathrm{Ker}d_0^Y\ar@{==}[rd]\ar[rr]^{} & &0\ar[r]_{} &0\ar[r]&\cdots\\
 & & & \mathrm{Ker}d_0^Y \ar@{-->}[ur]
  }$$

Let $K_{0\supset}= \cdots\rightarrow P_2\rightarrow P_1\rightarrow \mathrm{Ker}\pi\rightarrow 0$ and $K_{\subset0}= 0\rightarrow W\rightarrow P_{-1}\rightarrow P_{-2}\rightarrow\cdots$. Then there is a short exact sequence of complexes $0\longrightarrow K_{0\supset}\longrightarrow K\longrightarrow K_{\subset0}\longrightarrow 0.$
Let $T$ be any complex in $ex\widetilde{\mathcal{P}}$. Note that $T$ is totally acyclic. Then it follows from
\cite[Lemma 2.5]{CFH06} that the complex $\mathrm{Hom}_{R}(T, K_{\subset0})$ is exact, and this implies that $K_{\subset0}\in (ex\widetilde{\mathcal{P}})^{\perp}$. Note that $K_{0\supset}$ is an exact complex. For any morphism
$f: T\rightarrow K_{0\supset}$, we consider the following diagram:
$$\xymatrix@C=40pt{
\cdots \ar[r] &T_{2}\ar[d]_{f_2}\ar[r]^{} &T_{1}\ar[d]_{f_1}\ar[r]^{}\ar@{-->}[ld]_{s_1} &T_{0}\ar[r]\ar[d]_{f_{0}}\ar@{-->}[ld]_{s_0}
&T_{-1}\ar[r]\ar[d]^{}\ar@{-->}[ld]_{s_{-1}} &\cdots \\
\cdots \ar[r] &P_{2}\ar[r] &P_{1}\ar[r]^{}  &\mathrm{Ker}\pi \ar[r]&0\ar[r]&\cdots
  }$$
Let $s_i=0$ for any $i< 0$. Since $d_1^K: P_1\rightarrow \mathrm{Ker}\pi$ is an epic and $T_0$ is a projective module, there is a map
$s_0: T_0\rightarrow P_1$ such that $f_0 = d_{1}^{K}s_0$. Since $d_{1}^{K}(f_{1} - s_{0}d_{1}^{T}) = d_{1}^{K}f_{1} - d_{1}^{K}s_{0}d_{1}^{T} = d_{1}^{K}f_{1} - f_{0}d_{1}^{T} = 0$, then $f_{1} - s_{0}d_{1}^{T}: T_1\rightarrow \mathrm{Ker}d_{1}^{K}$, and there exists a map $s_1: T_1\rightarrow P_2$ such that $f_{1} - s_{0}d_{1}^{T} = d_{2}^{K}s_{1}$. Analogous to comparison theorem, we inductively get homotopy maps $\{s_i\}$ such that $f$ is null homotopic. Then $K_{0\supset}\in (ex\widetilde{\mathcal{P}})^{\perp}$. Thus, we have $K\in (ex\widetilde{\mathcal{P}})^{\perp}$.
Note that for any object in the model category $(\mathrm{Ch}(R), \mathcal{M}_{sing}^{ctr})$, its cofibrant replacement is precisely a special $ex\widetilde{\mathcal{P}}$-precover. Then it follows from the short exact sequence $0\rightarrow K\rightarrow P\stackrel{q}\rightarrow G(Y)\rightarrow 0$ that $P$ is a cofibrant replacement of $G(Y)$, and we can set $QG(Y) = P$.

Note that $F(K)= \cdots \rightarrow 0\rightarrow W\rightarrow 0\rightarrow\cdots$. Since $W$ is a module of finite projective dimension, for any complex $E\in ex\widetilde{\mathcal{I}}$, $\mathrm{Hom}_{R}(W, E)$ is exact. This implies that $F(K)\in {^{\perp}}(ex\widetilde{\mathcal{I}})$.
For $F(K)$, there is an exact sequence $0\rightarrow F(K)\rightarrow I\rightarrow L\rightarrow 0$ with $I\in ex\widetilde{\mathcal{I}}$ and $L\in {^{\perp}}(ex\widetilde{\mathcal{I}})$. We consider the following push-out diagram:
$$\xymatrix@C=20pt@R=20pt{ & 0\ar[d] & 0\ar[d] \\
0 \ar[r]^{}  &F(K) \ar[d] \ar[r] & F(P) \ar@{-->}[d]_{i}
  \ar[r]^{F(q)} &FG(Y) \ar@{=}[d]  \ar[r] &0 \\
0 \ar[r] & I \ar@{-->}[r] \ar[d] & J \ar[r]^{p} \ar[d] & GF(Y) \ar[r] & 0\\
  & L \ar[d] \ar@{=}[r] & L\ar[d]\\
  & 0 & 0
  }$$
It is clear that $i$ is a trivial cofibration. By the left column, we have $I\in {^{\perp}}(ex\widetilde{\mathcal{I}})$. Then $I\in ex\widetilde{\mathcal{I}}\cap {^{\perp}}(ex\widetilde{\mathcal{I}})$, and $p$ is a trivial fibration. Hence $F(q) = pi$ is a weak equivalence in $\mathcal{M}_{sing}^{co}$.
\end{proof}

\subsection*{The proof of the theorem}
It follows from Proposition \ref{prop 1} that $(F, G): (\mathrm{Ch}(R), \mathcal{M}_{sing}^{ctr})\longrightarrow (\mathrm{Ch}(R), \mathcal{M}_{sing}^{co})$ is a Quillen adjunction.
By \cite[Corollary 1.3.16]{Hov99}, there is a useful criterion for checking the given Quillen adjunction is a Quillen equivalence. Specifically, we need to show that $F$ reflects weak equivalences between cofibrant objects in $\mathcal{M}_{sing}^{ctr}$ (i.e. complexes in $ex\widetilde{\mathcal{P}}$), see Lemma \ref{lem 3}; moreover, for every fibrant object $Y$ in $\mathcal{M}_{sing}^{co}$ (i.e. $Y\in ex\widetilde{\mathcal{I}}$), we need to show that the composition $FQG(Y)\stackrel{F(q)}\rightarrow FG(Y)\stackrel{\varepsilon}\rightarrow Y$ is a weak equivalence, where $\varepsilon$ is the counit of the adjunction $(F, G)$, and $q: QG(Y)\rightarrow G(Y)$ is a cofibrant replacement of $G(Y)$, see Lemma \ref{lem 4} and \ref{lem 5}.

\begin{ack*}
The author is supported by National Natural Science Foundation of China (11871125), Natural Science Foundation of Chongqing (cstc2018jcyjAX0541) and the Science and Technology Research Program of Chongqing Municipal Education Commission (No. KJQN201800509).
\end{ack*}

\end{document}